\documentclass[review=false,hidelinks,onefignum,onetabnum]{siamart250211}

\usepackage{amssymb}
\usepackage{nicematrix}
\usepackage{bm}
\usepackage{relsize}
\usepackage{amscd}
\usepackage{url}
\usepackage[numbers,square,sort]{natbib}
\usepackage{doi}
\usepackage{graphicx,subcaption}
\usepackage{xcolor}
\usepackage{amsfonts}
\usepackage[tableposition=top]{caption}
\newtheorem{thm}{Theorem} 
\newtheorem{lem}{Lemma}

\makeatletter
   \renewcommand{\thetable}{%
   \thesection.\arabic{table}}
   \@addtoreset{table}{section}
  \makeatother

\newcommand{\matr}[1]{\begin{bmatrix} #1 \end{bmatrix}}
\newcommand{\R}{\mathbb{R}}
\newcommand{\Rm}[2]{\mathbb{R}^{#1 \times #2}}

\usepackage[normalem]{ulem}

\headers{Error bounds for approximate eigenvalues}{I. B. Haas and Y. Nakatsukasa}

\title{Sharp error bounds for approximate 
eigenvalues 
and singular values from subspace methods \thanks{Submitted to the editors on June 1st, 2025.}} 

\author{Irina-Beatrice Haas\footnotemark[2]
\and Yuji Nakatsukasa\thanks{Mathematical Institute, University of Oxford, OX2 6GG, UK (\email{haas@maths.ox.ac.uk}, \email{nakatsukasa@maths.ox.ac.uk}).} }

\ifpdf
\hypersetup{
  pdftitle={error bounds Ritz values},
  pdfauthor={I.-B. Haas, Y. Nakatsukasa}
}
\fi

\begin{document}

\maketitle

\begin{abstract}
Subspace methods are commonly used for finding approximate eigenvalues and singular values of large-scale matrices. 
Once a subspace is found, the Rayleigh-Ritz method (for symmetric eigenvalue problems) and Petrov-Galerkin projection (for singular values) are the de facto method for extraction of eigenvalues and singular values. 
In this work we derive quadratic error bounds for approximate eigenvalues of symmetric matrices obtained via the Rayleigh-Ritz process. Our bounds take advantage of the fact that extremal eigenpairs tend to converge faster than the rest, hence having smaller residuals $\|A\widehat x_i-\theta_i\widehat x_i\|_2$, where $(\theta_i,\widehat x_i)$ is a Ritz pair (approximate eigenpair). 
The proof uses the structure of the perturbation matrix underlying the Rayleigh-Ritz method to bound the components of its eigenvectors. In this way, we obtain a bound of the form 
$c\frac{\|A\widehat x_i-\theta_i\widehat x_i\|_2^2}{\mbox{Gap}_i}$, where $\mbox{Gap}_i$ is roughly the gap between the $i$th Ritz value and the eigenvalues that are not approximated by the Ritz process, and $c> 1$ is a modest scalar. 
Our bound is adapted to each Ritz value and is robust to clustered Ritz values, which is a key improvement over existing results. We further show that the bound is asymptotically sharp, and generalize it to singular values of arbitrary real matrices.
Finally, we apply these bounds to several methods for computing eigenvalues and singular values, and illustrate the sharpness of our bounds in  a number of computational settings, including Krylov methods and randomized algorithms.
\end{abstract}



\section{Introduction} \label{sec:intro}

The symmetric eigenvalue problem and the Singular value decomposition (SVD) are key computational tasks in many numerical methods for engineering and data analysis applications. Classical algorithms perform these decompositions in polynomial time (cubic in the dimension, for an $\epsilon$-accurate solution), but the matrices that are considered nowadays are often of very large scale, sometimes even too large to fit into memory. 
Subspace methods (including Krylov subspace methods \cite{Parlett98} and randomized algorithms \cite{HMT}) form a leading class of methods that allow us to tackle such problems. 
Once a subspace $\mbox{Span}(Q)$ is identified, where $Q\in\mathbb{C}^{n\times k}$ is a matrix with orthonormal columns, 
the most common 
approach to extracting eigenvalues is the Rayleigh-Ritz (RR) process \cite[Ch.~11]{Parlett98}, which outputs the eigenvalues of $Q^TAQ$ as approximations to those of $A$. 
The quality of the Ritz eigenpairs as an approximation to the actual eigenpairs of the original matrix has been extensively studied in the literature: for the eigenvectors, a bound on the canonical (or \textit{principal}) angle between the subspace spanned by the Ritz vectors and the actual eigenspace is provided in the classical Davis-Kahan sin $\theta$ theorem \cite{davisKahan}\cite[Chap.11]{Parlett98} and was improved in \cite{Yuji_sharpVect}. For the Ritz values, error bounds have been derived in \cite{LiLi,Jia2001,Parlett98}\cite[Cor.7.3.5]{Horn2012}\cite[Cor.1.4.31]{Stewart1998}. 
A similar process is applicable to approximate the (usually leading) singular values of $A\in\mathbb{C}^{m\times n}$ from subspace(s) spanned by $Q_1\in\mathbb{C}^{m\times k_1},Q_2\in\mathbb{C}^{n\times k_2}$, as the singular values of the projected matrix $Q_1^TAQ_2$; this is called a \emph{Petrov-Galerkin (PG) projection} method. 
More recently, randomized algorithms such as randomized SVD by Halko, Martinson and Tropp (HMT) \cite{HMT} and (Generalized) Nyström \cite{GNystrYuji,clarkson2009,woolfe2008} have been receiving significant attention, as they provide an efficient way of computing a low-rank approximation of the original matrix  that is nearly optimal in the Frobenius norm \cite{HMT}.  Similarly to RR and PG, randomized SVD is also based on (one-sided) projection, $A\approx QQ^TA$,
where $\mbox{Span}(Q)=\mbox{Span}(A\Omega)$ for a random (sketch) matrix $\Omega$. 
Error bounds on the resulting 
singular values for these methods are  derived in \cite{gu2015subspace,lazzarino,saibaba2019randomized}. 

In this paper we first examine the accuracy of the output approximate eigenvalues from the RR process. 
Our work is primarily motivated by the fact that existing error bounds for the Ritz eigenvalues can be improved greatly, assuming that the norms of the residuals corresponding to each eigenvalue differ significantly, which is common in practice. 
That is why we derive sharp upper bounds on the difference between the Ritz values and the exact eigenvalues by exploiting the perturbation structure underlying the Rayleigh-Ritz process. 
We then generalize the result for singular values of arbitrary matrices, 
and compare the bound to the approximate singular values obtained from HMT in our numerical experiments, along with other approximation methods including subspace iteration and Krylov methods. 

To make the problem precise, let us  fix the notation and consider the structure afforded by the RR and PG processes.

\textbf{Notation :} We use MATLAB notation for matrix indexing, in which $X(:,i)$ denotes the $i$th column of $X$ and $X(:,i:j)$ is the matrix consisting of the $i$th to the $j$th columns of $X$. If not stated otherwise, we also note $X_i$ the $i$th column of the matrix $X$.  
We let $\|X\|_2$ denote the spectral norm of a matrix $X$, which reduces to the Euclidean norm when $X$ is a vector.
For simplicity we focus on the real symmetric case (for the eigenvalue problem). 
Our analysis holds more generally verbatim for complex Hermitian matrices; in this case, one  replaces $^T$ with the Hermitian transpose $^*$ in what follows.
Finally, the error bounds presented below often involve spectral gaps that we denote $\mbox{Gap}_i, \mbox{gap}_i$ or $\Gamma_i$.  By definition $\mbox{Gap}_i$ and $ \Gamma_i$ will involve the part of the spectrum that is not approximated and the eigenvalue whose error is bounded, while $\mbox{gap}_i$ will involve (roughly) the approximated eigenvalue and the next closest  exact eigenvalue. Therefore 
typically $\mbox{Gap}_i>\mbox{gap}_i$, and we emphasize the case where $\mbox{Gap}_i\gg \mbox{gap}_i$. 
We also use the letter $\delta$ to denote available lower bounds on the spectral gaps.

\paragraph{Perturbed matrix and available information from the Rayleigh-Ritz process}

Suppose that $k$ eigenpairs of an $n\times n$ symmetric matrix $A$ are sought (typically  the smallest or largest eigenpairs), and we have a $k$-dimensional\footnote{In practice one often uses a subspace whose dimension is larger than the desired number of eigenpairs $k$, but this does not affect what follows. }
trial subspace spanned by an $n\times k$ matrix $Q_1$ with orthonormal columns that approximates the corresponding eigenspace. 
To extract an approximation to the desired eigenpairs from $Q_1$  with RR one computes the $k\times k$ symmetric eigendecomposition $Q_1^{T}AQ_1=Y\widehat\Lambda Y^T$
where $\widehat \Lambda=\mbox{diag}(\theta_1,\ldots,\theta_k)$ is the matrix of Ritz values and $Y$ is unitary (see e.g., \cite[Ch.~11]{Parlett98}). 
The matrix of Ritz vectors $\widehat X=[\widehat x_1,\ldots,\widehat x_k]$ is defined as $\widehat X=Q_1Y$. Then let $\widehat X_\perp$ be such that $[\widehat X\ \widehat X_\perp]$ is a unitary matrix, and form the \textit{perturbation matrix}
\begin{equation}  \label{eq:RRstruct} 
\bar{A} =[\widehat X\ \widehat X_\perp]^{T}A[\widehat X\ \widehat X_\perp]=
\begin{bmatrix}
\widehat\Lambda& E^{T}\\ E&A_2
\end{bmatrix} 
\end{equation}
where $E=\widehat X_\perp^{T}A\widehat X=[E_1,E_2,\ldots,E_k]$. The matrix $\bar A$ is similar to $A$ so it has the same eigenvalues as $A$.
Throughout the paper we let $\lambda_i$ (resp. $\sigma_i$) denote the $i$th eigenvalue (resp. singular value) of the matrix $A$. 
The overarching goal is to bound $|\theta_i-\lambda_i|$, the error in the Ritz values as approximations to the exact eigenvalues $\lambda_i$.

In practice, the information available after RR are the Ritz pairs $(\theta_i,\widehat x_i)$ for $i=1,\ldots, k$ and the norms of the columns of $E$, because they are equal to the norm of the residual: $\|E_i\|_2=\|A\widehat x_i-\theta_i \widehat x_i\|_2$. 
In addition, since $Q_1$ is designed to approximate a $k$-dimensional eigenspace, a rough estimate of the eigenvalues of $A_2$ is usually available. For example, when $Q_1$ approximates the smallest eigenspace we can reasonably expect $\lambda(A_2)\gtrsim \max_i\theta_i$ (this must hold in the limit $E\rightarrow 0$), where $\lambda(A_2)$ denotes the spectrum of $A_2$. 
Another option is to use oversampling, i.e. to approximate $k+\ell$ eigenvalues (or singular values) so that the $\ell$ additional Ritz values provide a good approximation of $\mbox{Gap}_i :=\min_j |\theta_i-\lambda_j(A_2)|\approx \min_{k<j\leq k+\ell} |\theta_i-\theta_j|$. 
More generally to obtain an error bound on $\theta_i$ we only need to assume the knowledge of a lower bound $\widehat \delta_i$ on $\mbox{Gap}_i$.

A key observation here is that 
the residuals $\|E_i\|_2$ typically vary with $i$ by several orders of magnitude. For example, if the smallest $k$ eigenpairs are sought and the Ritz values $\theta_i$ are arranged in increasing order, then we typically have 
a \emph{graded} structure 
$\|E_1\|_2\lesssim \|E_2\|_2\lesssim\ldots \lesssim \|E_k\|_2$, with $\|E_1\|_2\ll \|E_k\|_2$, because extremal eigenvalues tend to converge faster than interior ones by widely used methods such as Krylov subspace methods (like Lanczos) and LOBPCG \cite{Bai2000,LOBPCG_paper} (and HMT for singular values); this is related to the convergence of the power method. In this situation we observe that the error $|\theta_i-\lambda_i|$ is $O(\|E_i\|^2_2)$. Our goal is to derive bounds that accurately reflect this for each eigenpair.

Turning to the general case, to approximate $k$ singular values of a general matrix $A\in\mathbb{R}^{m \times n}$, 
using the PG method, assume we have a left trial subspace spanned by 
$Q_1\in \R^{m\times k}$ 
and a right trial subspace spanned by $Q_2 \in \R^{n\times k_2}$, with $Q_1, Q_2$ orthogonal,
and $k_2\geq k$ without loss of generality.  Take the SVD decomposition of $Q_1^T A Q_2= X \widehat{\Sigma} Y^T$ with $\widehat{\Sigma}=\mbox{diag}(\theta_1, \theta_2, \ldots, \theta_k)$ containing the approximate singular values, 
then multiply $A$ by the appropriate orthogonal matrices on both sides as follows to obtain
\begin{equation}\label{eq:RR_svd_structure}
    \bar{A} = \matr{X^T &  \\  & I_{m-r}}\matr{Q_1^T \\ Q_{1\perp}^T} A \matr{Q_2 & Q_{2\perp}} \matr{Y &  \\  & I_{n-r}} = \matr{\widehat{\Sigma} & E^T \\ F & A_2}.
\end{equation}
Note that when $k_2>k$, the first $k_2-k$ columns of $E^T$ are zero; indeed $k_2=n$ is allowed (as will be in HMT as we will see), in which case $E=0$. 
The perturbation matrix $\bar{A}$ is equivalent to $A$ up to orthogonal multiplication, so it has the same singular values as $A$.  Similarly to the symmetric case, we assume that the available information is the residual norms $\|E_i\|_2= \|A^T\hat u_i - \theta_i\hat v_i\|_2$ and $\|F_i\|_2= \|A\hat v_i - \theta_i\hat u_i\|_2$, where $F_i$ and $E_i$ are the columns of $E$ and $F$ respectively, and some gap information on the approximate singular values $\theta_i$ and the singular values of $A_2$.

Below we review existing results that use 
similar information, namely the residuals and some (approximate) gap information, to derive error bounds for singular values or eigenvalues.

\paragraph{Literature review}
Fundamentally, the problem set out in \Cref{eq:RRstruct} is that of eigenvalue perturbation: how do the eigenvalues of $\begin{bmatrix}
\widehat\Lambda& \\ &A_2\end{bmatrix}$ change by the off-diagonal perturbation $E$? A number of results are available in this direction, which we review here. A similar problem was considered in~\cite{LiAndYuji}, which  investigates the special case of a multiple Ritz value, ie. when $\widehat \Lambda= \mu I_k$, but the problem considered in this paper is for a general set of Ritz values. 

The simplest and most general (i.e. making no use of the perturbation structure) symmetric eigenvalue or singular value perturbation bound is provided by Weyl’s theorem \cite[Cor.7.3.5]{Horn2012}\cite[Cor.I.4.31]{Stewart1998}, which states that
\begin{equation}
    |\theta_i-\lambda_i| \leq \|E_i\|_2,
\end{equation}
in the symmetric case, and in the general case
\begin{equation}
    |\theta_i-\sigma_i| \leq \max\{\|E_i\|_2, \|F_i\|_2\}.
\end{equation}
However this only gives a linear bound (with respect to the residual $\|E_i\|_2$) although usually in practice the accuracy of $\theta_i$ is much higher.  An extensive line of work such as \cite{RQbounds} and the references therein offer improved linear bounds.
Other authors  \cite{LiLi,lazzarino,Nakatsukasa2012} have derived quadratic bounds that depend on the spectral norm of the error blocks from the perturbation matrix. 
Improving the bound from Mathias~\cite{Mathias1998}, C.-K. Li and R.-C. Li \cite{LiLi} derive an error bound for the case where the perturbation matrix is as in \Cref{eq:RRstruct,eq:RR_svd_structure}. 
For symmetric matrices, using again $\mbox{Gap}_i=\min_j|\theta_i-\lambda_j(A_2)|$ as above, they prove the bound: 
\begin{equation}\label{eq:liliSymbd}
    |\lambda_i-\theta_i| \leq \dfrac{2\|E\|_2^2}{\mbox{Gap}_i+ \sqrt{\mbox{Gap}_i^2+4 \|E\|_2^2} }
\end{equation}
and similarly, for singular values of general matrices: 
\begin{equation}\label{eq:liliSVDbd}
    |\sigma_i-\theta_i| \leq \dfrac{2\max \{\|E\|_2, \|F\|_2\}^2}{\mbox{Gap}_i+ \sqrt{\mbox{Gap}_i^2+4 \max \{\|E\|_2, \|F\|_2\}^2} },
\end{equation}
with $\mbox{Gap}_i=\min_j|\theta_i-\sigma_j(A_2)|$.

Alternatively, \cite{Nakatsukasa2012, lazzarino} give more general bounds since their perturbation matrix also potentially has nonzero components in the diagonal blocks. In \cite{lazzarino} the authors apply such bound to the structure afforded from Generalized Nyström, HMT and RR and compare the accuracy of these methods for singular value estimation. With our notation and considering only off diagonal perturbations, the bound from \cite[Theorem 4.1]{lazzarino} is :
\begin{equation}\label{eq:lorenzoSVDbd}
    |\sigma_i-\theta_i| \leq \dfrac{2\max \{\|E\|_2, \|F\|_2\}^2}{\min_k |\sigma_i-\sigma_k(A_2)|-2\max \{\|E\|_2, \|F\|_2\}}.
\end{equation}
\Cref{eq:liliSymbd,eq:liliSVDbd,eq:lorenzoSVDbd} are relevant in cases  where the individual residuals are not available, and \cite{LiLi} showed that their bound is asymptotically sharp as $\|E\|_2\longrightarrow 0$ when $\mbox{Gap}_i$ and $\|E\|_2$ are the only available information. 

Recalling from above that we expect $|\theta_i-\lambda_i|$ to be $O(\|E_i\|^2_2)$, 
the above bounds 
clearly overestimate the error for the extremal Ritz values.

Other existing quadratic bounds on the accuracy of Ritz values  that make use of the individual residuals include the bounds in \cite{Jia2001}, which uses the angle between $Q_1$ and an exact eigenspace, and the classical result 
\begin{equation}  \label{eq:classicalbd}
|\theta_i-\lambda_i|\leq \frac{\|E_i\|_2^2}{\widetilde{\mbox{gap}}_i}
\end{equation}
from \cite[Thm.~11.7.1]{Parlett98}, where $\widetilde{\mbox{gap}}_i$ is the gap between $\theta_i$ and the eigenvalues of $A$ excluding the one closest to $\theta_i$.  When the 
Ritz values are well separated, the bound \Cref{eq:classicalbd} is much sharper than the previous ones, especially assuming we have graded residuals, but this bound is not usable for clustered singular values and, similarly to \cite{Yuji_sharpVect}, we will show in this paper that the denominator in \Cref{eq:classicalbd} can be made larger. 

One can also use \Cref{eq:liliSymbd} from \cite{LiLi} (or alternatively the bound from \cite{Nakatsukasa2012}) with a $1$-by-$1$ block partitioning as follows: the perturbation matrix $\bar A$ can be rearranged as
\begin{equation} \label{eq:11block}
    \begin{bmatrix}
    \theta_i & & E_i^T \\
     & \widehat\Lambda_{\backslash i} & E_{\backslash i}^T \\
      E_i &  E_{\backslash i} & A_2
\end{bmatrix} =: \begin{bNiceArray}{c|cc}[margin]
\theta_i & 0_{1\times (k-1) }  & E_i^T \\
\hline
0_{(k-1)\times 1}  & \Block{2-2}{B} & \\
E_i & & 
\end{bNiceArray},
\end{equation}
where $\backslash i$ indicates that we removed only the elements corresponding to the $i$-th Ritz value from the blocks $E$ and $\widehat \Lambda$. This
 gives the bound
 \begin{equation}\label{eq:lili_11block}
 |\theta_i-\lambda_i| \leq  \dfrac{\|E_i\|^2_2}{\mbox{gap}_i +\sqrt{\mbox{gap}_i^2+ 4\|E_i\|_2^2}} 
 \end{equation}
 where the denominator $\mbox{gap}_i$  is the gap between  $\theta_i$ and the eigenvalues of the $(n-1)\times (n-1)$ matrix $B$ defined above. The drawback of this approach is that it puts the remaining Ritz values in the second diagonal block, and thus it reduces $\mbox{Gap}_i$ in \Cref{eq:liliSymbd,eq:liliSVDbd} to approximately the same denominator as in \Cref{eq:classicalbd}. Hence, another way of looking at the contribution made in this paper is that we improve the spectral gap in the error bounds of Ritz values from (roughly) $\mbox{gap}_i= \min_{j\neq i}|\theta_i-\lambda_j(A)|$ to $\mbox{Gap}_i$, while keeping the denominator as small as $\|E_i\|_2^2$.

\paragraph{Key contributions}
Motivated by the above discussion, in this paper we derive bounds of the form
\begin{equation} \label{eq:goal} 
|\theta_i-\lambda_i|\leq
  c\frac{\|E_i\|^2_2}{\mbox{Gap}_i}, 
\end{equation}
where $c$ is a modest constant such that $ c\rightarrow 1$ as $E\rightarrow 0$. Our proposed bounds are particularly good in  the case where the extremal eigenvalues are sought and the residuals are graded (but these conditions are not required for the bounds to be applicable). In this context, we show that our bounds are significantly sharper than those in the literature, including \cite{LiLi,Parlett98}, accurately reflecting the error observed in practice. Moreover, contrarily to \Cref{eq:classicalbd}, our bound \Cref{eq:goal} is also suitable for clustered eigenvalues, as $\mbox{Gap}_i$ is much larger than the denominator $\widetilde{\mbox{gap}}_i$ in \Cref{eq:classicalbd}. 

Since the constant $c$ above converges to $1$ as $E\rightarrow 0$,  we have 
\[\lim_{E\rightarrow 0}\frac{|\theta_i-\lambda_i|}{\|E_i\|^2_2} \leq \frac{1}{\min_j|\theta_i-\lambda_j(A_2)|}\leq \frac{1}{\widehat\delta_i},\]
and we show that this asymptotic bound is sharp, that is, it cannot be improved without more information about the eigenvalues of $A_2$.

The main message of this work is that when $E$ is sufficiently small, the error in a Ritz value $\theta_i$ is bounded by the square of the corresponding residual divided by the gap between $\theta_i$ and the eigenvalues of $A_2$, which are roughly the eigenvalues that are not sought; they converge when $E\rightarrow 0$. A practical implication is that in large-scale symmetric eigenvalue problems, working with a subspace of dimension larger than $k$ (which is often done for the purpose of avoiding missing eigenvalues) helps improve the accuracy of the desired ones, because doing so increases the denominator in the error bound \Cref{eq:goal}, which implies that a larger tolerance becomes acceptable for the residual  $\|E_i\|_2=\|A\widehat x_i-\theta_i \widehat x_i\|_2$ to ensure a given required accuracy of $\theta_i$.

The paper is organized as follows: in \Cref{sec:prelim} we describe the method that we use to obtain our theoretical error bounds for the eigenvalues of symmetric matrices in \Cref{sec:sym_bd} and for singular values of arbitrary matrices in \Cref{sec:svd_bd}. We then illustrate our theoretical bounds with numerical experiments in  \Cref{sec:experiments} and apply them for estimates obtained via RR, HMT, PG and Lanczos.


\section{Basic approach}\label{sec:prelim}
In this section we explain the main ideas that we use to derive the theorems in \Cref{sec:sym_bd}. 
We first recall a well-known result on the derivative of simple eigenvalues \cite{Stewart90}.
\begin{lem}\label{lem0}
Let $A_0$ and $F$ be symmetric matrices. 
Denote by $\lambda_i(t)$ the $i$th eigenvalue of $A_0+tF$ such that
 $(A_0+tF)x(t)=\lambda_i(t)x(t)$ where $\|x(t)\|_2=1$ for $t\in [0,1]$. If $\lambda_i(t)$ is simple, then 
\begin{equation}\label{pert}
\frac{d\lambda_i(t)}{d t}= x(t)^\ast F x(t).
\end{equation}
\end{lem}

In our analysis below we will have 
$F$ of the form 
$F=\big[\begin{smallmatrix}0& F_i^T\\F_i& 0  \end{smallmatrix}\big]$, where 
$F_i=[0_{(n-k)\times(i-1)}\ E_i \ 0_{(n-k)\times(n-i)}]$ has just one ($i$th) nonzero column. 
Then with the partitioning $x(t)=\begin{bmatrix}x_{1:k}(t)\\y(t)\end{bmatrix}$ (where for a vector $x$, we denote by $x_j$ its $j$th entry and by $x_{j:\ell}$ with $\ell\geq j$ the vector $[x_j,x_{j+1},\ldots,x_\ell]^T$), 
if $\lambda_i(t)$ is simple for $0\leq t\leq 1$ then 
by the above lemma we get 
\begin{align}
|\lambda_i(0) -\lambda_i(1)| &
=\left|\int_0^1 x(t)^T F x(t)dt\right|  \label{smally}\\
&\leq 2\left|\int_0^1 y(t)^T F_i x_{1:k}(t)dt\right| \nonumber\\
&= 2\left|\int_0^1 y(t)^T E_i x_i(t)dt\right| \nonumber\\
&\leq 2\|E_i\|_2\left|\int_0^1 \|y(t)\|_2dt\right|.\qquad (\|x_i(t)\|_2\leq 1)\label{smally2}
\end{align}
Note that in the setting of~\eqref{eq:RRstruct}, we have 
 $\lambda_i(0)=\theta_i$ and $\lambda_i(1)$ is an eigenvalue of $A$, so \eqref{smally2} provides an error bound for the 
Ritz value $\theta_i$. 
The key observation here 
 is that \eqref{smally2} is small  if $\|y(t)\|_2$ is small for all $0\leq t\leq 1$. In view of this, our approach below is to obtain sharp bounds for 
$\|y(t)\|_2$, from which we get sharp bounds for $|\lambda_i(0) -\lambda_i(1)|$ by \eqref{smally2}.

The idea of obtaining eigenvalue perturbation bounds via bounding eigenvector components was introduced in \cite{Nakatsukasa2012}. Moreover, it is shown there that \eqref{smally2} holds even in the presence of multiple eigenvalues, in which case $x(t)$ can be taken as any of the (many possible) eigenvectors associated with $\lambda_i(t)$, and the bounds hold as long as the bound on $\|y(t)\|_2$ holds for any eigenvector corresponding to the eigenvalue, which is the case in the forthcoming analysis. 
Hence in what follows we are not concerned with whether $\lambda_i(t)$ is simple or not.


\section{Error bounds for Ritz values} \label{sec:sym_bd}
We are now ready to derive error bounds on the Ritz values of symmetric matrices. 
First we derive an error bound for a single Ritz value that depends on the corresponding residual norm $\|E_i\|_2$, then we use the same approach to derive a bound for a set of clustered eigenvalues based on the spectral norm of the corresponding submatrix of $E$. 

\subsection{Error bound for well separated Ritz values} 
In this subsection we prove the following theorem, which is the main result of this paper.
\begin{thm}\label{thm:main}
Let 
\[A=\begin{bmatrix}
\widehat\Lambda& E^{\ast}\\ E&A_2
\end{bmatrix}\]
be a symmetric matrix where $\widehat \Lambda=\mbox{diag}(\theta_1,\ldots,\theta_k)$ 
and $E=[E_1,E_2,\ldots, E_k]$. 
For $i=1,2,\ldots, k$, suppose that 
\begin{align}
\delta_i&=\min_j|\theta_i-\lambda_j(A_2)|-\|E_i\|_2
\ (={\rm Gap}_i-\|E_i\|_2)
>0, \label{del1}\\
\delta_{i,j}&=|\theta_i-\theta_j|-\|E_i\|_2>0,  \quad \mbox{for}\quad 1\leq j\leq k,\  j\neq i \label{del2}
\end{align}
 and further that 
 \begin{equation}   \label{eq:cicondition}
d_i=\frac{1}{\delta_i-\mathlarger{\sum}_{1\leq j\leq k\atop j\neq i}\frac{\|E_j\|_2^2}{\delta_{i,j}}}>0.   
 \end{equation}
Then there exists an eigenvalue $\lambda_i$ of $A$ satisfying
\begin{equation}  \label{eq:mainbound}
|\lambda_i -\theta_i|\leq  d_i\|E_i\|_2^2. 
\end{equation}
The Ritz values $\theta_i$ and eigenvalues $\lambda_i$ have a one-to-one correspondence. 
Moreover, let $x(t), \lambda_i(t)$ be continuous functions of $t\in [0,1]$ defined by
\[
\begin{bNiceArray}{ccc|c}[margin]
\widehat \Lambda_{(1:i-1)} & & & E_{1:i-1}^T \\
& \theta_i &  & tE_i^T \\
& & \widehat \Lambda_{(i+1:k)} & E_{i+1:k}^T \\
\hline
E_{1:i-1} & tE_i & E_{i+1:k} & A_2 \\
\end{bNiceArray}x(t) = \lambda_i(t) x(t)  
\]
and $\lambda_i(0)=\theta_i$, where we note $\widehat\Lambda_{(a:b)}=\widehat \Lambda(a:b,a:b)$. Then the bound~\eqref{eq:mainbound} holds for any $\delta_i,\delta_{i,j}$ such that 
$0<\delta_i\leq \min_{j,0\leq t\leq 1}|\lambda_i(t)-\lambda_j(A_2)|$ and $0<\delta_{i,j}\leq \min_{0\leq t\leq 1}|\lambda_i(t)-\theta_j|$.
\end{thm}

\begin{proof}
Let $i\leq k$ be an integer. 
First, we decompose the matrix $A$ from the theorem as
\begin{equation}
    A = \begin{bmatrix}\widehat\Lambda& \\ &A_2 \end{bmatrix}+\begin{bmatrix}
         & E^T \\ E & 
    \end{bmatrix}=:A_0+\widehat E.
\end{equation}
As above let
$F=\big[\begin{smallmatrix}0& F_i^T\\F_i& 0  \end{smallmatrix}\big]$, where 
$F_i=[0_{(n-k)\times(i-1)}\ E_i\ 0_{(n-k)\times(n-i)}]$ 
is the matrix obtained by taking the $i$th row and column of $\widehat E$. 
Let $\widetilde E$ be a matrix such that $\widehat E=F+\widetilde E$, so that 
$A=A_0+\widehat E=A_0+\widetilde E+F$. Note that $\theta_i$ is an eigenvalue of $A_0+\widetilde E$. 
Suppose that $(A_0+\widetilde E+tF)x(t)=\lambda_i(t)x(t)$ for $0\leq t\leq 1$, and define $y(t)=x_{k+1:n}(t)$ 
where $\lambda_i(t)$ is a continuous function of $t$ satisfying $\lambda_i(0)=\theta_i$. 
Our goal is to bound $|\lambda_i-\theta_i|=|\lambda_i(1)-\lambda_i(0)|$ using \eqref{smally2}, so we aim to derive a bound for $\|y(t)\|_2$. 

To achieve this, we next exploit the structure of the perturbation matrix $A_0+\Tilde{E}+tF$ as follows.
For any $j$ such that $j\leq k$ and $i\neq j$, the $j$th row of
 $(A+\widetilde E+tF)x(t)=\lambda_i(t)x(t)$ is 
\[\theta_jx_j(t)+E_j^\ast y(t)=\lambda_i(t)x_j(t),\]
hence 
\begin{equation}  \label{eq:xi2}
|x_j(t)|=\frac{|E_j^\ast y(t)|}{|\theta_j-\lambda_i(t)|}\leq \frac{\|E_j\|_2\|y(t)\|_2}{|\theta_j-\lambda_i(t)|}.  
\end{equation}
The last $n-k$ rows of $(A_0+\widetilde E+tF)x(t)=\lambda_i(t)x(t)$ give
\[
A_2y(t)+\sum_{1\leq j\leq k\atop j\neq i}E_jx_j(t)+tE_ix_i(t)=\lambda_i(t)y(t).
\]
Hence 
\[tE_ix_i(t)=(\lambda_i(t)I-A_2)y(t)-\sum_{1\leq j\leq k\atop j\neq i}E_jx_j(t),\]
so
\begin{align*}
t\|E_i\|_2|x_i(t)|&\geq \|(\lambda_i(t)I-A_2)y(t)-\sum_{1\leq j\leq k\atop j\neq i}E_jx_j(t)\|_2  \\
&\geq \|(\lambda_i(t)I-A_2)y(t)\|_2-\sum_{1\leq j\leq k\atop j\neq i}
\|E_jx_j(t)\|_2  \\
&\geq \|(\lambda_i(t)I-A_2)y(t)\|_2-\sum_{1\leq j\leq k\atop j\neq i}
\frac{\|E_j\|_2^2\|y(t)\|_2}{|\theta_j-\lambda_i(t)|}, 
\end{align*}
where we used \eqref{eq:xi2} to get the last inequality. 
Defining $\Gamma_i(t)=\min_j|\lambda_i(t)-\lambda_j(A_2)|$ 
(note that $\Gamma_i(t)$ is closely related to $\mbox{Gap}_i$; in particular $\Gamma_i(0)=\mbox{Gap}_i$) 
we have $\|(\lambda_i(t)I-A_2)y(t)\|_2\geq \Gamma_i(t)\|y(t)\|_2$, so it follows that 
\begin{align}
\left(\Gamma_i(t)-\sum_{1\leq j\leq k\atop j\neq i}\frac{\|E_j\|_2^2}{|\theta_j-\lambda_i(t)|} \right)\|y(t)\|_2
&\leq t\|E_i\|_2|x_i(t)|. \label{eq:A}
\end{align}
Suppose that there exist positive scalars $\delta_i,\delta_{i,j}$ ($i\neq j$) such that 
\[\delta_i\leq\min_{0\leq t\leq 1}\Gamma_i(t),\quad 
\delta_{i,j}\leq \min_{0\leq t\leq 1}|\theta_j-\lambda_i(t)|.
\]
For example, by Weyl's theorem we can take 
 $\delta_i=\Gamma_i(0)-\|E_i\|_2=\min_j|\theta_i-\lambda_j(A_2)|-\|E_i\|_2$ and 
$\delta_{i,j}= |\theta_j-\lambda_i(0)|-\|E_i\|_2=|\theta_j-\theta_i|-\|E_i\|_2$, provided that they are both positive. 
Thus $\Gamma_i(t)-\sum_{1\leq j\leq k\atop j\neq i}\frac{\|E_j\|_2^2}{|\theta_j-\lambda_i(t)|}\geq  \delta_i-\sum_{1\leq j\leq k\atop j\neq i}\frac{\|E_j\|_2^2}{\delta_{i,j}}$ for $0\leq t\leq 1$, and if the right-hand side is positive then defining 
\begin{equation}  \label{eq:deljineq}
d_i:=\frac{1}{\delta_i-\sum_{1\leq j\leq k\atop j\neq i}\frac{\|E_j\|_2^2}{\delta_{i,j}}}\ (>0), 
\end{equation}
from \eqref{eq:A} we have 
\[
\|y(t)\|_2\leq td_i\|E_i\|_2|x_i(t)|\leq td_i\|E_i\|_2,
\]
where we used $|x_i(t)|\leq \|x(t)\|_2=1$. 
Thus we have a bound for $\|y(t)\|_2$. 
Plugging this into \eqref{smally2} we obtain
\begin{align*}
|\lambda_i(1) -\lambda_i(0)|&
\leq 2\|E_i\|_2\left|\int_0^1 \|y(t)\|_2dt\right|\\
&\leq 2\|E_i\|_2\left|\int_0^1td_i \|E_i\|_2dt\right|\\
&\leq d_i\|E_i\|_2^2.
\end{align*}
Moreover the bound~\eqref{eq:mainbound} holds for any $\delta_i,\delta_{i,j}$ such that $0<\delta_i\leq \min_{j,0\leq t\leq 1}|\lambda_i(t)-\lambda_j(A_2)|$ and $0<\delta_{i,j}\leq \min_{0\leq t\leq 1}|\lambda_i(t)-\theta_j|$.
\end{proof}

The fact that $\theta_i$ and $\lambda_i$ have a one-to-one correspondence was not explicitly stated nor necessary in our derivation of \Cref{thm:main}, but this can be verified by noting that $\lambda_i(t)$ can be taken as the $i$th eigenvalue of $A_0+ \Tilde{E}+tF$ and by defining an appropriate ordering, e.g. by keeping the non-increasing order. Moreover this question is irrelevant when the residuals are small enough as (for well separated Ritz values) the error intervals become disjoint. 
The possible presence of multiple eigenvalues does not affect these arguments for the reason discussed at the end of \Cref{sec:prelim}.

From \Cref{thm:main} we see that our bound is of the form \eqref{eq:goal}, as announced in the introduction. 
In a typical application where the extremal eigenvalues are sought, 
in the limit $E\rightarrow 0$ we have $\theta_i\rightarrow \lambda_i$ for $i=1,\ldots, k$, where $\lambda_i$ is arranged in appropriate (increasing or decreasing) order. 
Furthermore $\min(\lambda(A_2))=\lambda_{k+1}$, hence we have $\mbox{Gap}_i\rightarrow|\lambda_i-\lambda_{k+1}|$, and $d_i$ in \eqref{eq:mainbound} approaches $1/(\lambda_{k+1}-\lambda_i)$. 
It follows that \eqref{eq:mainbound} gives the asymptotic bound 
\begin{equation}  \label{eq:asympt}
\lim_{E\rightarrow 0}\frac{|\lambda_i -\theta_i|}{\|E_i\|_2^2}\leq\frac{1}{|\lambda_{k+1}-\lambda_i|}.
\end{equation}
This asymptotic bound is tighter than any of the known bounds; for example the bound in~\cite{LiLi} is 
$\lim_{E\rightarrow 0}\frac{|\lambda_i -\theta_i|}{\|E\|_2^2}\leq\frac{1}{|\lambda_{k+1}-\lambda_i|}$ (the left-hand side is smaller than in~\eqref{eq:asympt}), 
and the classical bound on a single Ritz value~\cite[Thm.~11.7.1]{Parlett98} gives 
$\lim_{E\rightarrow 0}\frac{|\lambda_i -\theta_i|}{\|E_i\|_2^2}\leq\frac{1}{\min_{j}|\lambda_{j}-\lambda_i|}$ (the right-hand side is larger than in~\eqref{eq:asympt}).  
The lower bounds $\delta_i,\delta_{i,j}$ of $\min_j|\lambda_i(t)-\lambda_j(A_2)|, \min_{0\leq t\leq 1}|\lambda_i(t)-\theta_j|$  as defined in \eqref{del1}, \eqref{del2} use Weyl's theorem and so are crude bounds. However, as confirmed in our numerical experiments, this does not affect the asymptotic sharpness of the theorem in the limit $E\rightarrow 0$, because we still have $d_i\rightarrow \frac{1}{\min_j|\theta_i-\lambda_j(A_2)|}$.

Finally, the asymptotic bound \eqref{eq:asympt} is sharp, provided that no further information on $A_2$ is available. In order to prove this we can show that for the specific case where $A_2=cI$ with $c>\theta_i$ for all $i=1, \ldots, k$, in the limit $E\longrightarrow 0$ we have 
\begin{equation}\label{eq:expansion}
    |\lambda_i - \theta_i| = \dfrac{\|E_i\|_2^2}{\mbox{Gap}_i} + O(\|E_i\|_2^4) 
\end{equation}
To see this, take a fixed $i$ and rearrange the columns and rows of $A-\theta_i I$, with $A$ from \Cref{thm:main}, such that the (1,1) coefficient is zero. This leads to the matrix
\[B =\left[
\begin{array}{c|ccc|c}
    0 & 0 & \cdots & 0 & E_i^T \\
    \hline
    0 & \theta_1-\theta_i & & & E_1^T \\
    \vdots & & \ddots & & \vdots \\
    0 & & & \theta_k-\theta_i & E_k^T \\
    \hline
    E_i & E_1 & \cdots & E_k & (c-\theta_i) I    
\end{array}
\right].
\]
Define $\Tilde{A}_2=B(2:n,2:n)$. Then ~\cite[Thm.~3.1]{LiAndYuji} states that the smallest eigenvalue of the above matrix is $|\lambda_i-\theta_i | = \left|\matr{0_{1\times (k-1)} & E_i^T}\Tilde{A}_2^{-1}\matr{0_{(k-1)\times 1} \\ E_i}\right|+ O(\|E_i\|^4_2/\mbox{Gap}_i^2)$. Assuming that the residuals $E_k, k\neq i$, are small enough the leading terms arising from this expression (e.g. by Neumann series) yields \Cref{eq:expansion}.

\subsection{Error bound for a cluster of Ritz values} \label{sec:cluster}
The bound derived above is sharp in the limit $E\rightarrow 0$, but for finite $E$ we might not have
$\delta_i-\sum_{1\leq j\leq k\atop j\neq i}\frac{\|E_j\|_2^2}{\delta_{i,j}}>0$ or $\delta_{i,j}>0$ for some $j$, 
in which case Theorem~\ref{thm:main} is not applicable. 
Assuming $E$ is not too large, one sees that this can happen only if $\delta_{i,j}\lesssim  \min_{0\leq t\leq 1}|\theta_j-\lambda_i(t)|\approx 0$ 
for some $j$. 
This means there is a cluster of Ritz values and $\theta_i$ belongs to it. 
In \Cref{thm:maincluster} we give a bound that is applicable in such situations. 
Much of the analysis is the same as above. 

\begin{thm}\label{thm:maincluster} 
Let 
\[A=\begin{bmatrix}
\widehat\Lambda& E^{\ast}\\ E&A_2
\end{bmatrix}\]
be a symmetric matrix where $\widehat \Lambda=\mbox{diag}(\theta_1,\ldots,\theta_k)$ 
and $E=[E_1,E_2,\ldots, E_k]$, and suppose that $\ell$ Ritz values $\theta_i,\ldots,\theta_{i+\ell-1}$ 
all lie in the interval $[\lambda_0-\Delta,\lambda_0+\Delta]$. Define ${\cal I}=\{i,i+1,\ldots,i+\ell-1\}$ and 
let $E_{\cal I}=E_{i:i+\ell-1}$ and 
suppose also that 
\begin{align*}
\delta_{\cal I}&=\min_j|\lambda_0-\lambda_j(A_2)|-\Delta-\|E_{\cal I}\|_2 \ (\approx \mbox{Gap}_{\mathcal{I}}-\Delta-\|E_{\mathcal{I}}\|_2)>0,\\
\delta_{{\cal I},j}&=  |\theta_j-\lambda_0|-\Delta-\|E_{\cal I}\|_2>0, \quad \mbox{for}\quad 1\leq j\leq k,\  j\notin {\cal I}
\end{align*}
 and further that 
\[d_{\cal I}=\frac{1}{\delta_{\cal I}-\mathlarger{\sum}_{1\leq j\leq k\atop j\notin {\cal I}}\frac{\|E_j\|_2^2}{\delta_{{\cal I},j}}}>0.\]
Then there exist $\ell$ eigenvalues $\lambda_i,\ldots, \lambda_{i+\ell-1}$ of $A$ satisfying
\begin{equation}  \label{eq:mainboundcluster}
|\lambda_{i+j} -\theta_{i+j}|\leq  d_{\cal I}\|E_{\cal I}\|_2^2, \quad j=0,1,2,\ldots,\ell-1. 
\end{equation}
Moreover, for $j=0,1,\ldots,\ell -1$ ,let $x(t), \lambda_{i+j}(t)$ be continuous functions of $t\in [0,1]$ defined by
\[
\begin{bNiceArray}{ccc|c}[margin]
\widehat \Lambda_{(1:i-1)} & & & E_{1:i-1}^T \\
& \widehat\Lambda_{\mathcal{I}} &  & tE_{\mathcal{I}}^T \\
& & \widehat \Lambda_{(i+\ell:k)} & E_{i+\ell:k}^T \\
\hline
E_{1:i-1} & tE_{\mathcal{I}} & E_{i+\ell:k} & A_2 \\
\end{bNiceArray}x(t) = \lambda_{i+j}(t) x(t).
\]
and $\lambda_{i+j}(0)=\theta_{i+j}$, where we note $\widehat\Lambda_{(a:b)}=\widehat \Lambda(a:b,a:b)$. Then the bound~\eqref{eq:mainboundcluster} holds for any $\delta_{\mathcal{I}},\delta_{\mathcal{I},j}$ such that 
$0<\delta_{\mathcal{I}}\leq \min_{i\in \mathcal{I},j}\min_{0\leq t\leq 1}|\lambda_i(t)-\lambda_j(A_2)|$ and $0<\delta_{\mathcal{I},j}\leq \min_{i\in \mathcal{I}}\min_{0\leq t\leq 1}|\lambda_i(t)-\theta_j|$.
\end{thm}

\begin{proof}
    Suppose that  $\ell$ Ritz values  form a cluster and assume 
without loss of generality that their indices are also clustered in  ${\cal I}=\{i,i+1,\ldots,i+\ell-1\}$, so 
the cluster consists of  $\theta_i,\ldots,\theta_{i+\ell-1}$  all lying in the interval
 $[\lambda_{0}-\Delta,\lambda_0+\Delta]$ (if not, then we can apply
 a permutation to cluster the indices).

As before define $y(t)=x_{k+1:n}(t)$ and let $E_{\cal I}=E_{i:i+\ell-1}$. 

Let $F=
\big[\begin{smallmatrix}
0& F_{\cal I}^T\\F_{\cal I}& 0
\end{smallmatrix}\big]$, where 
$F_{\cal I}=[0_{(n-k)\times(i-1)}\ E_{\cal I}\ 0_{(n-k)\times(n-i-\ell+1)}]$ has $\ell$ nonzero columns, and 
let $\widetilde E$ be a matrix such that $E=F+\widetilde E$. 
Let $\hat i$ be an arbitrary index in ${\cal I}=\{i,i+1,\ldots,i+\ell-1\}$. 
For any $j$ such that $j\leq k$ and $j\notin {\cal I}$, the $j$th row of 
 $(A+\widetilde E+tF)x(t)=\lambda_{\hat i}(t)x(t)$ is 
\[\theta_jx_j(t)+E_j^\ast y(t)=\lambda_{\hat i}(t)x_j(t),\]
hence 
\begin{equation}  \label{eq:xi2cluster}
|x_j(t)|=\frac{E_j^\ast y(t)}{|\theta_j-\lambda_{\hat i}(t)|}\leq \frac{\|E_j\|_2\|y(t)\|_2}{|\theta_j-\lambda_{\hat i}(t)|}.  
\end{equation}
Now, writing $x_{\cal I}=x_{i:i+\ell-1}$, from the last $n-k$ rows of
 $(A+\widetilde E+tF)x(t)=\lambda_{\hat i}(t)x(t)$ we get 
\[
tE_{\cal I}x_{\cal I}(t)+\sum_{1\leq j\leq k\atop j\notin {\cal I}}E_jx_j(t)+A_2y(t)=\lambda_{\hat i}(t)y(t).
\]
Hence
\[tE_{\cal I}x_{\cal I}(t)=(\lambda_{\hat i}(t)I-A_2)y(t)-\sum_{1\leq j\leq k\atop j\notin {\cal I}}E_jx_j(t),\]
so
\begin{align*}
t\|E_{\cal I}\|_2\|x_{\cal I}(t)\|_2&\geq \|(\lambda_{\hat i}(t)I-A_2)y(t)-\sum_{1\leq j\leq k\atop j\notin {\cal I}}E_jx_j(t)\|_2  \\
&\geq \|(\lambda_{\hat i}(t)I-A_2)y(t)\|_2-\sum_{1\leq j\leq k\atop j\notin {\cal I}}\frac{\|E_j\|_2^2\|y(t)\|_2}{|\theta_j-\lambda_{\hat i}(t)|}, 
\end{align*}
where we used \eqref{eq:xi2cluster} to get the last inequality.

Defining $\Gamma_{\cal I}(t)=\min_{i\in{\cal I},j}|\lambda_i(t)-\lambda_j(A_2)|$ we have 
 $\|(\lambda_{\hat i}(t)I-A_2)y(t)\|_2\geq \Gamma_{\cal I}(t)\|y(t)\|_2$, so  it follows that 
\begin{align*}
\left(\Gamma_{\cal I}(t)-\sum_{1\leq j\leq k\atop j\notin {\cal I}}\frac{\|E_j\|_2^2}{|\theta_j-\lambda_{\hat i}(t)|} \right)\|y(t)\|_2
&\leq t\|E_{\cal I}\|_2\|x_{\cal I}(t)\|_2, 
\end{align*}
Suppose that there exist positive scalars $\delta_{\cal I},\delta_{{\cal I},j}$ for  $j\notin {\cal I}$  such that 
\[\delta_{\cal I}\leq\min_{0\leq t\leq 1}\Gamma_{\cal I}(t),
\quad \delta_{{\cal I},j}\leq \min_{i\in{\cal I}}\min_{0\leq t\leq 1}|\theta_j-\lambda_i(t)|.\]
For example, by Weyl's theorem we can take
\begin{align*}
\delta_{\cal I}&=
\min_j|\lambda_0-\lambda_j(A_2)|-\Delta-\|E_{\cal I}\|_2,\\
\delta_{{\cal I},j}&= |\theta_j-\lambda_{\hat i}(0)|-\|E_{\cal I}\|_2\geq |\theta_j-\lambda_0|-\Delta-\|E_{\cal I}\|_2. 
\end{align*}
Thus $
\Gamma_{\cal I}(t)-\sum_{1\leq j\leq k\atop j\notin {\cal I}}\frac{\|E_j\|_2^2}{|\theta_j-\lambda_{\hat i}(t)|} 
\geq  \delta_{\cal I}-\sum_{1\leq j\leq k\atop j\notin{\cal I} }\frac{\|E_j\|_2^2}{\delta_{{\cal I},j}}$ for $0\leq t\leq 1$, 
and if the right-hand side is positive then defining 
\begin{equation}  \label{eq:deljineqcluster}
d_{\cal I}:=\frac{1}{\delta_{\cal I}-\sum_{1\leq j\leq k\atop j\notin {\cal I}}\frac{\|E_j\|_2^2}{\delta_{{\cal I},j}}} \ (>0)  ,
\end{equation}
we have 
\[
\|y(t)\|_2\leq td_{\cal I}\|E_{\cal I}\|_2\|x_{\cal I}\|_2
\leq td_{\cal I}\|E_{\cal I}\|_2. 
\]
Plugging this into \eqref{smally} 
yields 
\begin{align*}
|\lambda_{\hat i}(1) -\theta_{\hat i}|&
\leq 2\left|\int_0^1 x_{1:k}(t)^T E_{\cal I} y(t)dt\right|\\
&\leq 2\|E_{\cal I}\|_2\left|\int_0^1 t\|y(t)\|_2dt\right|
\qquad (\|x_{1:k}(t)\|_2\leq 1)\\
&\leq 2\|E_{\cal I}\|_2\left|\int_0^1td_{\cal I} \|E_{\cal I}\|_2dt\right|\\
&\leq d_{\cal I}\|E_{\cal I}\|_2^2. 
\end{align*}
The same argument holds for any $\hat i\in {\cal I}$, which proves the bound~\eqref{eq:mainbound}. Moreover, the bound~\eqref{eq:mainbound} holds for any
$\delta_{\cal I}, \delta_{{\cal I},j}$ such that 
$0<\delta_{\cal I}\leq \min_{0\leq t\leq 1}|\lambda_i(t)-\lambda_j(A_2)|$
 and $0<\delta_{{\cal I},j}\leq \min_{i\in{\cal I}}\min_{0\leq t\leq 1}|\lambda_i(t)-\theta_j|$, 
where $\lambda_i(t),\lambda_{i+1}(t),\ldots, \lambda_{i+\ell-1}(t)$ are continuous functions of
 $t$ such that $\lambda_{i+j}(t)$ is an eigenvalue of $A+\widetilde E+tF$ with
 $\lambda_{i+j}(0)=\theta_{i+j}$ for $j=0,1,\ldots,\ell-1$. 
\end{proof}
Conceptually, in the definition of $\delta_{\cal I}=\min_j|\lambda_0-\lambda_j(A_2)|-\Delta-\|E_{\cal I}\|_2$, the term 
$\min_j|\lambda_0-\lambda_j(A_2)|$ corresponds to the big $\mbox{Gap}_{\cal I}$ in previous discussions.


\section{Error bound for singular values}\label{sec:svd_bd}
In this section, recalling the setup of \Cref{eq:RR_svd_structure}, we generalize Theorem \ref{thm:main} to the singular values of arbitrary matrices of size $m$ by $n$. To do this, we use the Jordan-Wielandt theorem from ~\cite[Thm. I.4.2]{Stewart90}~\cite[Thm. 7.3.3]{Horn2012}, which allows us to 
apply the bound from the symmetric case to obtain results on the singular values. This theorem states that if $M$ is an $m$ by $n$ matrix with $m>n$ and singular values $\sigma_1\geq \sigma_2 \geq \ldots \geq \sigma_n$, then the eigenvalues of the symmetric matrix $\matr{0 & M \\ M^T & 0}$ are $\sigma_1\geq \sigma_2 \geq \ldots \geq \sigma_n\geq 0 \geq -\sigma_n \geq -\sigma_2\geq \ldots \geq -\sigma_1$, where the eigenvalue $0$ has multiplicity $m-n$. This classical result was also used in \cite{LiLi,lazzarino} for a similar purpose. 

This leads to the following theorem. 

\begin{thm}\label{thm_svd}
    Let \[
    A = \matr{\widehat\Sigma & E^T \\ F & A_2}
    \]
    be an $m$-by-$n$ matrix ($m\neq n$) where $\widehat\Sigma = \mbox{diag}(\theta_1, \ldots, \theta_k)$, $E =[E_1,E_2,\ldots, E_k]$ and $F =[F_1,F_2,\ldots, F_k]$. For $i=1,2,\ldots, k$, suppose that 

\begin{align*}
    \delta_i &= \min(|\theta_i|,\min_j |\theta_i - \sigma_j(A_2)|)-\sqrt{(\|E_{i}\|^2_2 + \|F_{i}\|^2_2)/2} >0 \\
    \delta_{i,j} &= |\theta_i-\theta_j |- \sqrt{(\|E_{i}\|^2_2 + \|F_{i}\|^2_2)/2} >0, \quad\text{for} \quad 1\leq j\leq k, \; j\neq i \\
    \delta'_{i,j} &= |\theta_i+\theta_j |- \sqrt{(\|E_{i}\|^2_2 + \|F_{i}\|^2_2)/2} >0, \quad\text{for} \quad 1\leq j\leq k
\end{align*}
and that
\begin{align*}
    d_i = \dfrac{1}{\delta_i - \mathlarger{\sum}_{1\leq j\leq k \atop j\neq i} \dfrac{\|E_j\|^2_2+\|F_j\|^2_2}{2\delta_{i,j}} -\mathlarger{\sum}_{1\leq j\leq k} \dfrac{\|E_j\|^2_2+\|F_j\|^2_2}{2\delta'_{i,j}}} > 0.
\end{align*}
Then there exists a singular value $\sigma_i$ of $A$ such that 
\[|\sigma_i -\theta_i|\leq  d_i\dfrac{\|E_{i}\|_2^2 + \|F_{i}\|_2^2}{2}.\]
\end{thm}

\begin{proof}
We consider the following matrix whose eigenvalues are $\{\sigma_i, -\sigma_i\}_{i=1, \ldots, n}$ and $0$, according to the Jordan-Wielandt theorem:
\[
\matr{  &   & \widehat{\Sigma} & E^T \\ 
      &   & F & A_2 \\
    \widehat{\Sigma} & F^T &   &   \\
    E & A_2^T &   &  }.
\]
As in \cite{LiLi, lazzarino}, by permuting the second and third block columns then the second and third block rows we obtain 
\[\Tilde{A}= \left[
\begin{array}{cc|cc}
    & \widehat{\Sigma} &   & E^T \\ 
    \widehat{\Sigma} &  & F^T &   \\
    \hline 
      & F &  & A_2 \\
    E &  & A_2^T & 
\end{array}\right] =: \matr{\Tilde{\Sigma} & E_{tot}^T \\ E_{tot} & \Tilde{A}_2}.
\]
Finally, using the unitary matrices 
$M=\frac{1}{\sqrt{2}}\matr{I_k & -I_k \\ I_k & I_k}$ and $P = \matr{M & 0 \\ 0 & I_{m+n-2k}}$ we have
\[
P\Tilde{A} P^T =  \left[
\begin{array}{cc|cc}
    -\widehat{\Sigma}& & \frac{1}{\sqrt{2}}F^T  & \frac{1}{\sqrt{2}}E^T \\ 
    & \widehat{\Sigma} & \frac{1}{\sqrt{2}}F^T & -\frac{1}{\sqrt{2}}E^T \\
    \hline 
    \frac{1}{\sqrt{2}}F & \frac{1}{\sqrt{2}}F &  & A_2 \\
    \frac{1}{\sqrt{2}}E & -\frac{1}{\sqrt{2}}E  & A_2^T & 
\end{array}\right]
\]
which is symmetric and of the form $B = \matr{\Tilde{\Lambda} & \Tilde{E}^T \\ \Tilde{E} & \Tilde{A}_2}$, with $\Tilde{E} = E_{tot}M^T = \frac{1}{\sqrt{2}}\matr{F & F \\ E & -E}$.

Then the Theorem \ref{thm:main} is directly applicable to $B$, which proves \Cref{thm_svd}. The slight change in the coefficient $\delta_i$ (compared to \Cref{thm:main}) comes from the fact that in addition to the $\pm \sigma_j(A_2)$ the matrix $\Tilde{A}_2$ also has the eigenvalue $0$, and the additional terms $\delta'_{i,j}$ appear from the block $-\widehat \Sigma$ from $\widetilde \Lambda$. %
\end{proof}

To compare our bound with \Cref{eq:liliSVDbd} from \cite{LiLi}, notice that the numerator in \Cref{eq:liliSVDbd} is always larger than the one from our bound since $\max \{\|E\|_2, \|F\|_2\}^2 \geq (\|E\|_2^2+\|F\|_2^2)/2\geq (\|E_i\|_2^2+\|F_i\|_2^2)/2$ for $i=1, \ldots, k$, while the denominator tend to $\mbox{Gap}_i$ in both bounds, which shows that asymptotically our bound is sharper, as announced in the introduction. 

The bound from \cite{lazzarino} with the structure considered in our \cref{thm_svd} gives \Cref{eq:lorenzoSVDbd} 
which is asymptotically looser than \Cref{eq:liliSVDbd} since we can assume $\min_k |\sigma_i-\sigma_k(A_2)|\approx \mbox{Gap}_i= \min_j |\theta_i - \sigma_j(A_2)|$ as $E\rightarrow0$. 
However we believe that \cite{lazzarino} would provide a better tool to analyze the error when the algorithm used for the approximate SVD 
has perturbation terms also in the diagonal blocks.
This was shown to be the case for the generalized Nyström approach in \cite[Sec.3]{lazzarino}.

In practice the residuals 
$E$ and $F$ may differ significantly (e.g. $E=0$ in HMT). 
However, our numerical experiments suggest that this does not affect the quality of the bound very much (see \Cref{sec:exp_svd}). Additionally, we derived another bound specifically for the HMT perturbation structure where $E=0$, by starting from a singular value perturbation result similar to \Cref{lem0} and bounding the components of the singular vectors. We obtained a bound that is theoretically slightly looser than \Cref{thm_svd}, so we omit this extension.


\section{Numerical experiments}\label{sec:experiments}

In this section we present several experiments to illustrate the validity and the sharpness of our theoretical bounds compared to previous existing bounds. We first test our error bounds for the eigenvalues of a (synthetic) symmetric matrix, then for singular values of arbitrary matrices. The latter subsection also discusses the application of our bound to the singular values obtained from the randomized SVD algorithm by Halko, Martinson and Tropp \cite{HMT}, 
which is so far the most widely used randomized method for low-rank approximation of large scale matrices.

\subsection{Numerical experiments on the eigenvalues of symmetric matrices} \label{sec:exp_sym}
We generate a symmetric matrix  $A=VDV^T$ of size $n=2000$
where $V\in \R^{n\times n}$ is orthogonal and $D$ is a diagonal matrix of eigenvalues. 
Suppose that we are looking for $k=100$ eigenpairs, and we run 40 steps of LOBPCG~\cite{LOBPCG_code} (for \textit{Locally Optimal Block Preconditioned Conjugate Gradient} method) without preconditioning\footnote{
An effective preconditioner can speed up the convergence of any eigenvalue, but it is usually still true that the extremal Ritz values converge first.} to obtain $k$ sets of Ritz pairs $(\theta_i,\widehat x_i)$ for $i=1,\ldots,k$. By default LOBPCG approximates the smallest eigenvalues. 
We next evaluate the errors in the Ritz values $|\theta_i-\lambda_i|$ and compare their bounds both in the context of well separated eigenvalues and for a matrix $D$ that has a cluster of eigenvalues.  For the initial guess $X_0$ we used a randomly generated  $n\times k$ Haar distributed matrix using MATLAB's {\tt orth(randn(n,k))}.

\paragraph{Well-separated Ritz values}
We let the eigenvalues be uniformly distributed with $D=\mbox{diag}(1,2,\ldots, n)$. 
To illustrate the sharpness of our theoretical bound we first use the exact value of $\mbox{Gap}_i$, that is, taking (the usually unknown quantity) $\lambda_j(A_2)$ as known. The resulting bound is shown as ``Thm.~\ref{thm:main} exact''. However in practice this distance is usually unknown, so to invoke Theorem~\ref{thm:main} using only information that is usually available, we use the estimate $\mbox{Gap}_i\geq \max_k |\theta_i-\theta_{k}|$, hence we let $\delta_i=\max_k |\theta_i-\theta_{k}|-\|E_i\|_2$ instead of \eqref{del2}. This assumption is valid when the eigensolver approximates the extreme eigenvalues and the approximation is high quality. 
The resulting bound corresponds to ``Thm.~\ref{thm:main} approximate'' in \Cref{fig:sym_fig_sep}. 
We compare our bound to the exact error $|\theta_i-\lambda_i|$, the bound \eqref{eq:liliSymbd} from~\cite{LiLi} (shown as ``Li-Li (large gap)''), the bound \eqref{eq:lili_11block} which is \eqref{eq:liliSymbd} applied to the partitioning~\ref{eq:11block} (``Li-Li (1-1 block)''), the ``classical" bound~\eqref{eq:classicalbd}, and $\|E_i\|_2$ which is the crude bound using Weyl's theorem. 
Notice from Weyl's bound that $\|E_1\|_2\lesssim \|E_2\|_2\lesssim \ldots \lesssim \|E_k\|_2$ with $\|E_1\|_2\ll \|E_k\|_2$, which is a typical graded behavior as mentioned in the introduction. 

\begin{figure}
    \centering
    \includegraphics[width=0.8\linewidth]{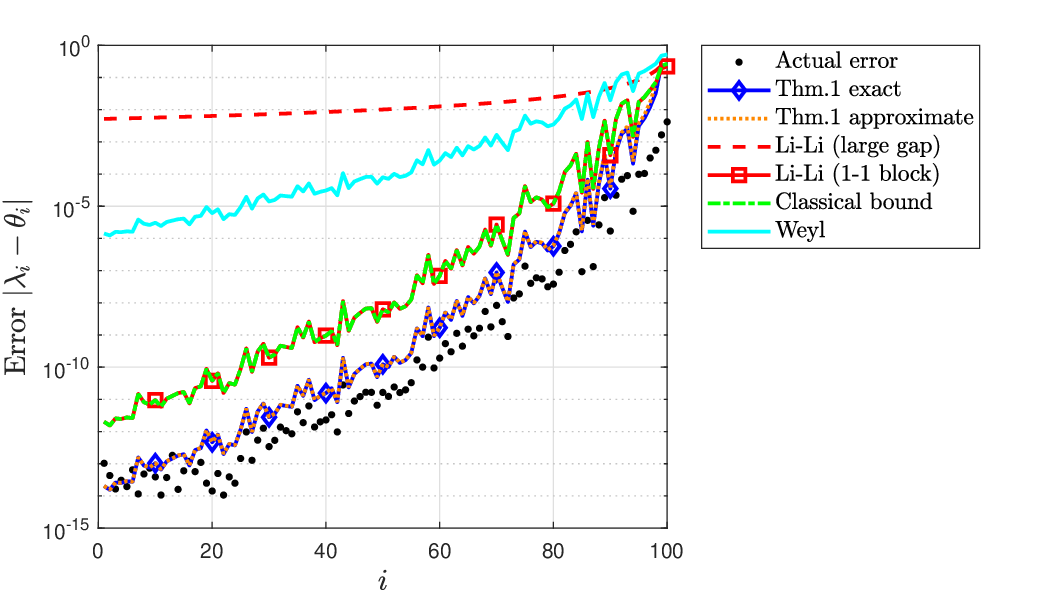}
    \caption{Error in Ritz values $|\theta_i-\lambda_i|$ and error bounds for uniformly distributed eigenvalues: $\lambda_i=i, \: \forall i\in [1,n]$.}
    \label{fig:sym_fig_sep}
\end{figure}

In \Cref{fig:sym_fig_sep} we observe that Theorem~\ref{thm:main} gives the sharpest error bound for the Ritz values (whether or not $\lambda(A_2)$ is known or estimated), especially for small $i$. 
The bounds Classical and Li-Li (1-1 block) are also of good quality and are nearly identical in this experiment. The latter and our theorem reflect the $\mathcal{O}(\|E_i\|_2^2)$ trend of the actual errors, unlike Weyl's theorem. For interior eigenvalues, our \Cref{thm:main} gives results close to Li-Li and Classical (because the spectral gaps in all three bounds are about the same); but as we move to extreme eigenvalues, the term $\mbox{Gap}_i$ in the denominator in \Cref{thm:main} increases (and the residual $\|E_i\|_2$ decreases), which leads to significantly higher accuracy (here, about a factor 100 improvement). This illustrates the key idea that our bounds are asymptotically sharp as $E\rightarrow 0$.

The bound ``Li-Li (large gap)'' mainly serves to illustrate that, if the residual in the numerator is $\|E\|_2$ the accuracy of the bound is significantly lower, even if the gap in the denominator is $\mbox{Gap}_i$. This is intuitive when the residuals are graded, like in this example.

Note that the two versions of our bounds (exact and approximate) are nearly identical, resulting in the two bounds being nearly superimposed in \Cref{fig:sym_fig_sep} (as well as in the other experiments below). In color print the exact bound (blue line) and the approximate bound (orange dots) are distinguishable but in black-and-white they might appear as the same line.
Hence we argue that Theorem~\ref{thm:main} can give tight bounds using information that is usually available in practice. 
By contrast we note that the Classical result~\eqref{eq:classicalbd} uses a bound on the gap between $\theta_i$ and the eigenvalues of an $(n-1)\times (n-1)$ matrix, so it is generally not a practical bound to use (although reasonable estimates of the gap can often be obtained). 

Another note of caution when using the bound $\|E_i\|_2$ (and possibly  Classical) is that the Ritz values and eigenvalues of $A$ may not have a 
one-to-one correspondence, as noted in~\cite[Sec.~11.5]{Parlett98}. Our results and Li-Li overcome this difficulty as it explicitly specifies a one-to-one correspondence between $\theta_i$ and $\lambda_i$.

\paragraph{Clustered Ritz values}
To illustrate the effectiveness of Theorem~\ref{thm:maincluster} in the presence of a clustered eigenvalue, 
we run the same experiment as above 
but now we replace 
the 20th to 29th eigenvalues of $A$ by $20+{\tt randn(1,10)*1e-10}$
 so that there are 10 eigenvalues clustered around $20$. 
Since the clustered Ritz values lied in $20\pm 10^{-10}$, 
we set ${\cal I}=\{20,\ldots,29\}$ 
with $\lambda_0=20, \Delta=10^{-10}$ when invoking Theorem~\ref{thm:maincluster} to bound the errors in $\theta_i$ for $i=20,\ldots,29$, and for the rest we let ${\cal I}=\{i\}$ with $\delta=0$, 
which reduces to Theorem~\ref{thm:main}. We also performed the same experiment where 10 eigenvalues of $A$ are clustered around $\lambda_0=100$ with $\Delta = 10^{-10}$, such that the Ritz values near $i=k$, ie. near $\lambda(A_2)$, are clustered. For both values of $\lambda_0$, in "Li-Li (1-1 block)" we used a partitioning similar to \eqref{eq:11block} that puts the clustered Ritz values in the 1-1 block and gives the residual $\|E_{\cal I}\|_2^2$ in the numerator of the bound. The results are presented in \Cref{fig:sym_fig_clust}.

\begin{figure}[htbp]
    \centering
    \begin{subfigure}[t]{0.49\textwidth}
        \centering
        \includegraphics[width=\linewidth]{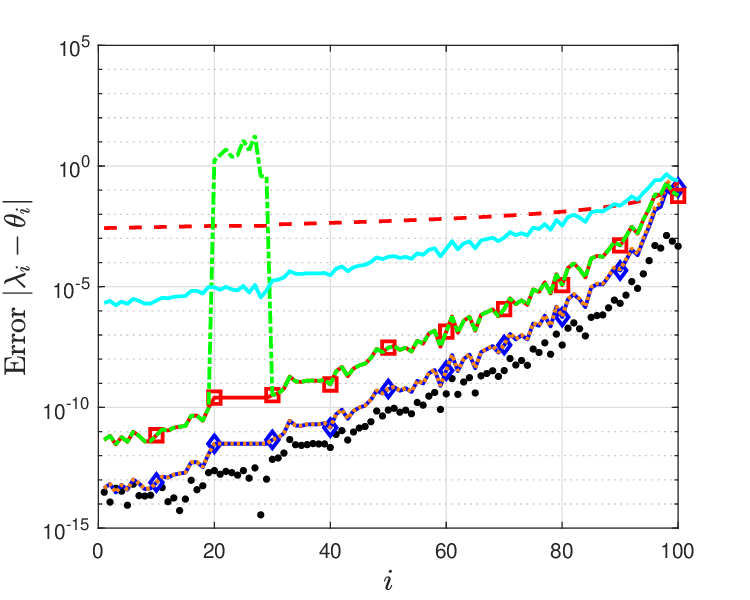}
        \caption{Cluster: $\lambda_i=20+10^{-10}\times\tt{randn(1)}, \: \forall i\in [20,29]$.}
        \label{fig:sym_fig_clust1}
    \end{subfigure}
    \hfill
    \begin{subfigure}[t]{0.49\textwidth}
        \centering
        \includegraphics[width=1\linewidth]{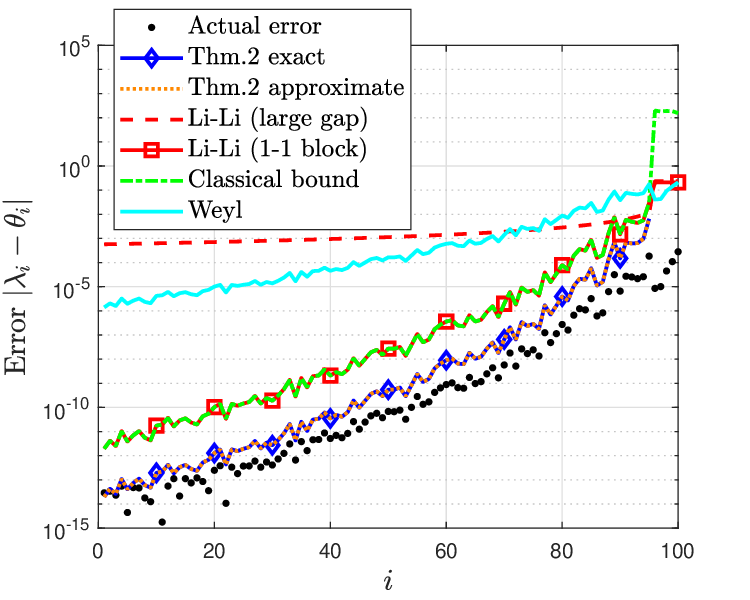}
        \caption{Cluster: $\lambda_i=100+10^{-10}\times\tt{randn(1)}, \: \forall i\in [96,105]$. 
        Note that our bounds are inapplicable for $i\geq 96$, as the assumptions $\delta_{\mathcal{I}}>0$ and $d_{\mathcal{I}}>0$ do not hold. }
        \label{fig:sym_fig_clust2}
    \end{subfigure}
    \caption{Error in Ritz values $|\theta_i-\lambda_i|$ and error bounds for uniformly distributed eigenvalues ($\lambda_i=i$) and a cluster of 10 eigenvalues at $\lambda_0=20$ (left) and $\lambda_0=100$ (right).} 
    \label{fig:sym_fig_clust}
\end{figure}

For the uniformly distributed eigenvalues, the same comments as in the previous experiment apply, therefore we focus our attention on the clustered eigenvalues. In both experiments the classical bound is highly inaccurate to bound the errors on the cluster (because $\widetilde{\mbox{gap}}_i$ in the denominator of \Cref{eq:classicalbd} is approximately $ 10^{-10}$). Remarkably, Li-Li (1-1 block) gives accurate results both when $\lambda_0=20$ and $\lambda_0=100$. Our \Cref{thm:maincluster} is more accurate than Li-Li when the cluster is well inside the set of approximated eigenvalues (see \Cref{fig:sym_fig_clust1}), but it is inapplicable in \Cref{fig:sym_fig_clust2}, where the Ritz values cluster together with some eigenvalues of $A_2$. On the other hand, the bound from \cite{LiLi} is still applicable, so in this situation it may be necessary to use \eqref{eq:lili_11block} instead of our \Cref{thm:maincluster}.

In practice, the main assumption that we need to check in order to decide if our bounds apply is $\|E_i\|_2^2 < \mbox{Gap}_i^2, \forall i=1,\ldots,k $. The small gaps $|\theta_i-\theta_j|$ between neighboring Ritz values are less of a limitation since we can decide to consider certain Ritz values as clustered to ensure that $\delta_{i,j}>0, \forall j\neq i$, as long as no cluster is not confounded with a part of $\sigma(A_2)$. Therefore, for fixed residual norms, our bounds are particularly efficient when the distribution of the eigenvalues of $A$ (especially for the eigenvalues that we seek to approximate) is steep. Nonetheless, when $E$ is sufficiently small relative to the gaps, our bound will give $c\dfrac{\|E_i\|_2^2}{\mbox{Gap}_i}$ with $c\approx 1$, which is tight as shown in~\eqref{eq:expansion}.

\paragraph{Numerical experiments with the Lanczos algorithm}
In this experiment, we apply the Lanczos algorithm \cite{Paige1976ErrorAO,Lanczos1950} (with re-orthogonalization) instead of LOBPCG to obtain a trial matrix of eigenvectors $Q_1$. 
By construction of $Q_1$ from the Lanczos iteration we have
\[
AQ_1 = Q_1 T_k + q_{k+1}[0, \ldots, 0, t_{k+1,k}]
\]
with $T_k\in \Rm{k}{k}$ a tridiagonal matrix, $q_{k+1}\in \mathbb{R}^n$ a vector which is orthogonal to the columns of $Q_1$ and with $t_{k+1,k} = \|v-v^TQ_1(:,k)-v^TQ_1(:,k-1)\|_2$ where $v=AQ_1(:,k)$. Then noting $Q_\perp$ the complement of $\matr{Q_1 & q_{k+1}}$ we have
\[
\Bar{A} = \matr{Q_1 & q_{k+1} & Q_\perp}^TA \matr{Q_1 & q_{k+1} & Q_\perp} =
\begin{bNiceArray}{cccc|cc}[margin]
\Block{4-4}<\Large>{T_k} & & & & 0 & 0 \\
& & & & \Vdots & \vdots \\
& & & & 0 & \vdots \\
& & & & t_{k+1,k} & 0 \\
\hline
0 & \Cdots& 0 & t_{k+1,k} & \Block{2-2}<\Large>{*} & \\
0 & \Cdots & \Cdots & 0 & &
\end{bNiceArray}
\] 
and using the eigendecomposition of the tridiagonal block $T_k=U\widehat\Lambda U^T$ we obtain the perturbation matrix
\[
\matr{U^T & \\ & I} \Bar{A} \matr{U & \\ & I} = 
\begin{bNiceArray}{c|cc}[margin]
\hat \Lambda & e & 0 \\
\hline
e^T & \Block{2-2}{A_2} & \\
0 & & 
\end{bNiceArray}
\]
with $e=U^T[0, \ldots, 0, t_{k+1,k}]^T= t_{k+1,k} U^T(:,k)$. Therefore we have a particular structure where each  residual vector $E_i$ only has one nonzero component, namely their first component.

As in LOBPCG, with the Lanczos algorithm usually the extreme eigenvalues converge faster than the interior eigenvalues. Note that $t_{k+1,k}$ can be relatively large, leading to large residuals $\|E_i\|$ for some $i$,  
therefore we have to take a relatively large number of Lanczos iterations to ensure that a sufficient number of Ritz values have small residuals. 

In our experiment, we take $n=2000$ and we use 
Lanczos with full reorthogonalization to compute the trial subspace $Q_1$. We set the size of the Krylov subspace to $k=400$.
Some $\theta_i$ have relatively large residuals therefore, to test our bound, we rearranged the structured matrix such that only the 20 smallest approximated eigenvalues are in the $(1,1)$ diagonal block and the other ones are included in the lower diagonal block. 
Note that adding the larger Ritz values in the block (2,2) should not change the gap $\mbox{Gap}_i$, as those eigenvalues are far from the ones approximated in the block (1,1).  The results are shown in \Cref{fig:lanczos}. 
Again our bounds are the sharpest for the smallest 
eigenvalues as $\|E_i\|_2\rightarrow 0$, and the approximate bound is close to the exact one.

\begin{figure}[h]
    \centering
    \includegraphics[width=0.5\linewidth]{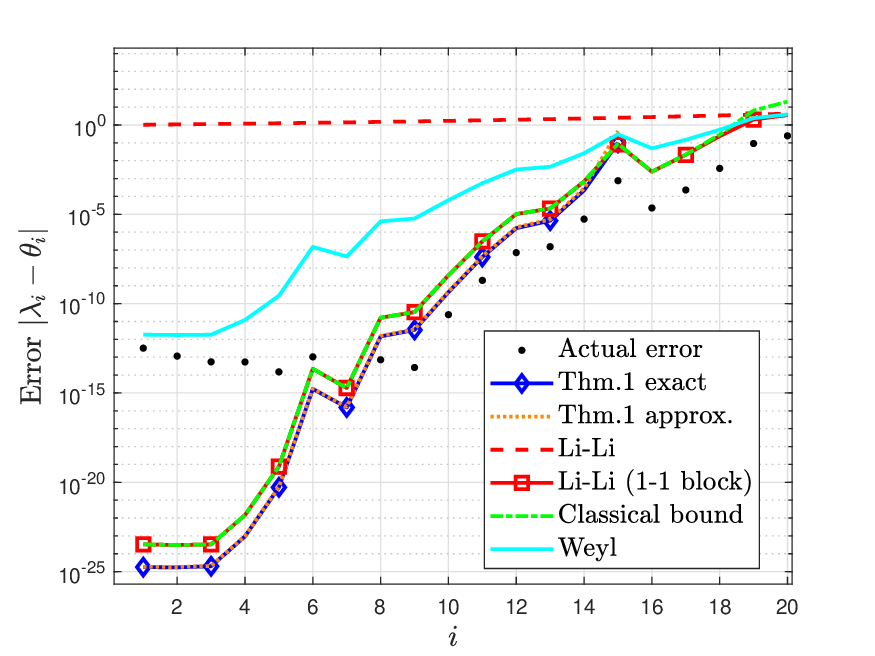}
    \caption{Error $|\theta_i-\lambda_i|$ and bounds for approximate eigenvalues obtained with the Lanczos algorithm. The eigenvalues of $A$ that were approximated here are $\lambda_i \in [1,20]$. Some data points are above the bounds but only because of the limited machine precision used in the experiment.}
    \label{fig:lanczos}
\end{figure}

\subsection{Numerical experiments for singular values obtained via the Rayleigh-Ritz and HMT method}\label{sec:exp_svd}
We now illustrate the quality of the error bound derived in \Cref{sec:svd_bd} for singular values of a $m$-by-$n$ matrix. We defined a matrix $A$ of size $m=5000$ by $n=1000$ 
with geometrically decaying singular values using MATLAB's command {\tt gallery('randsvd',[m,n],1e20)}. 
To find approximate singular subspaces $Q_1\in \R^{m\times k} ,Q_2\in \R^{n\times k}$ we used both the simple and the double power iteration. For instance for a single power iteration, we start from random normal matrices $\Omega_1, \Omega_2$ and take $Q_1 = \mbox{orth}(A\Omega_1)$ and $Q_2 = 
\mbox{orth}(\Omega_2^TA)$. 
For a double power iteration, instead of applying $A$ one applies $AA^TA$. We approximated the $k=200$ largest singular values with several methods.

To apply the error bounds from Theorem \ref{thm_svd} for the PG method we used the structure \eqref{eq:RR_svd_structure} detailed in \Cref{sec:intro}. 
We also applied this theorem to the case where the HMT algorithm is used to approximate the singular values. Let us derive the perturbation matrix associated to HMT. 
The main difference from PG in HMT  is that only the left trial subspace $Q_1$ of $A$ is used (i.e., one can think $Q_2=I_n$ in HMT). We kept the same trial subspace to plot the RR and the HMT results (in left and right plots of \Cref{HMT_and_RR_power1,HMT_and_RR_power2}) to make sure that the results are comparable. 
Then let $Q_1Q_1^TA$ be the HMT approximation of $A$. Taking the full SVD of the matrix $Q_1^TA=U_0 \matr{\Sigma_0 & 0} \matr{V_0^T \\ V_{0\perp}^T}$ and noting $Q_{tot} = [Q_1 \quad Q_{1\perp}]$, 
we see that the perturbation matrix underlying the HMT approximation can be written as 
\[
\matr{U_0^T & 0 \\ 0 & I_{m-r}} Q_{tot}^T  A \matr{V_0 & V_{0\perp}} = \matr{\Sigma_0 & 0 \\ Q_{1\perp}^T A V_0 & Q_{1\perp}^T A V_{0\perp} }
\]
which is a particular case of the structured matrix from \eqref{thm_svd}, with one of the error blocks being zero. 
Both in the PG method and HMT, the residual norms can be computed with the available information.

For both approximation methods, we compare our bounds (exact and approximate, as in the symmetric case) to the bound \eqref{eq:liliSVDbd} given in \cite{LiLi}, to \eqref{eq:lorenzoSVDbd} from \cite{lazzarino} and to Weyl's theorem ~\cite[Cor. 7.3.5]{Horn2012} ~\cite[Cor. I.4.31]{Stewart1998}. 
The results are shown in \Cref{HMT_and_RR_power1,HMT_and_RR_power2}. 

\begin{figure}[htbp]
    \centering
    \begin{subfigure}[t]{0.48\textwidth}
        \centering
        \includegraphics[width=\linewidth]{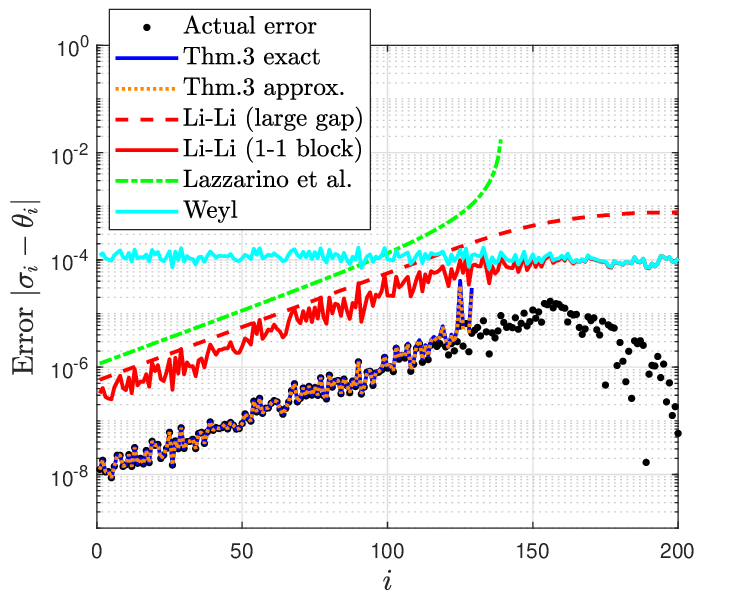}
        \caption{Petrov-Galerkin approximation}
    \end{subfigure}
    \hfill
    \begin{subfigure}[t]{0.48\textwidth}
        \centering
        \includegraphics[width=1\linewidth]{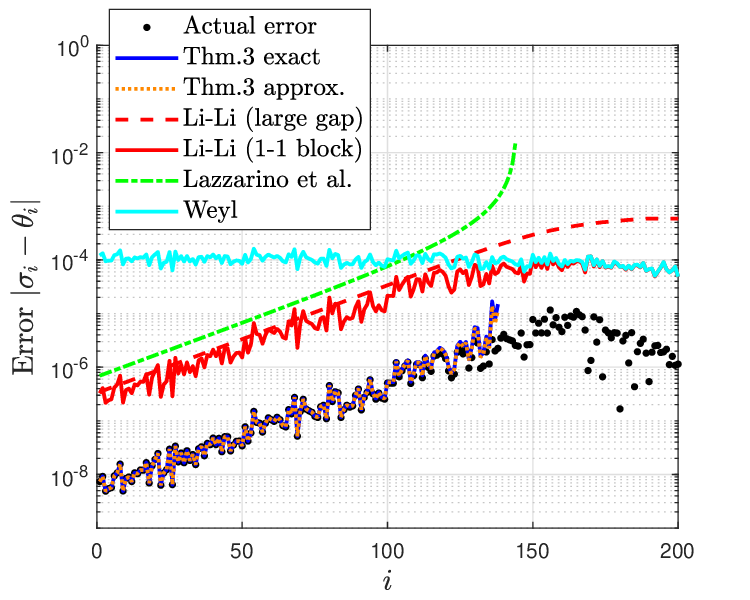}
        \caption{Randomized SVD}
    \end{subfigure}
    \caption{Error $|\sigma_i-\theta_i|$ and bounds for geometrically distributed singular values where the trial subspaces were found with a single power iteration. Left: estimation with Petrov-Galerkin approximation ; right: estimation with randomized SVD. 
    }
    \label{HMT_and_RR_power1}
\end{figure}

\begin{figure}[htbp]
    \centering
    \begin{subfigure}[t]{0.48\textwidth}
        \centering
        \includegraphics[width=\linewidth]{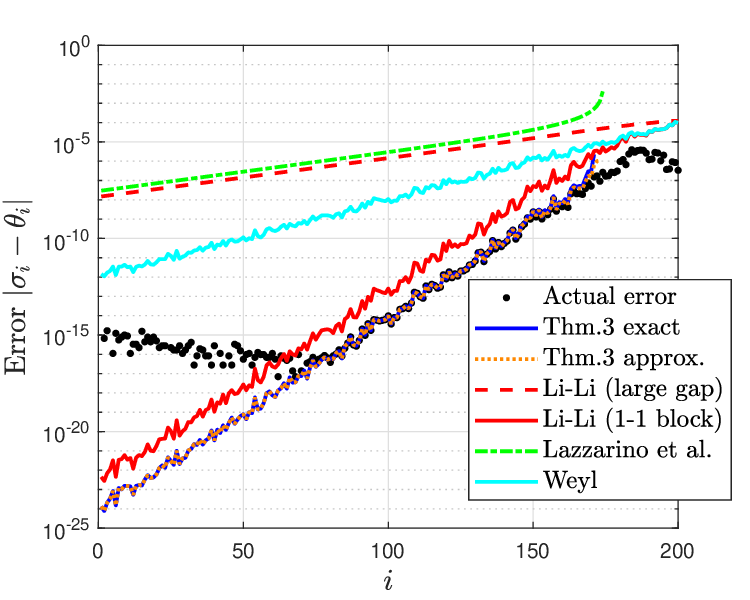}
        \caption{Petrov-Galerkin approximation}
    \end{subfigure}
    \hfill
    \begin{subfigure}[t]{0.48\textwidth}
        \centering
        \includegraphics[width=1\linewidth]{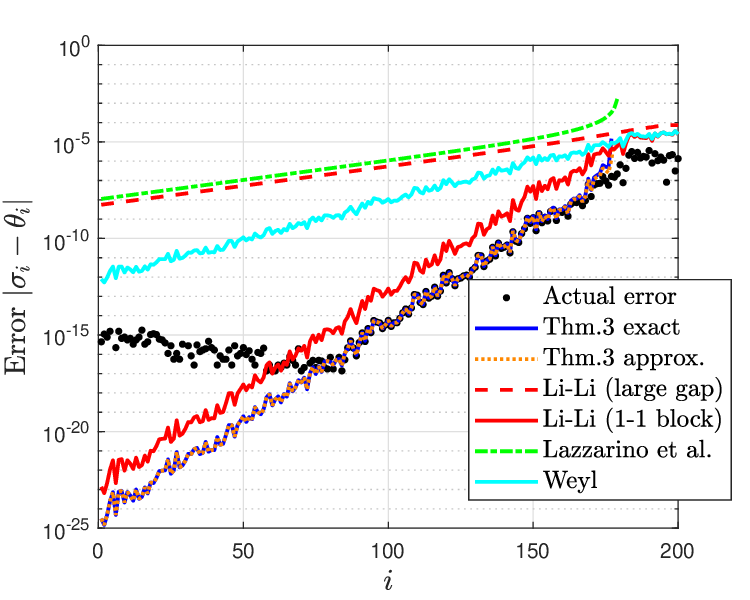}
        \caption{Randomized SVD}
    \end{subfigure}
    \caption{Error $|\sigma_i-\theta_i|$ and bounds for geometrically distributed singular values where the trial subspaces were found with double power iteration. Left: estimation with Petrov-Galerkin approximation ; right: estimation with randomized SVD.
    In some cases the bounds lie below the actual error; this is due to roundoff errors (which our bounds do not account for). Indeed, the effect of round-off errors was not accounted for in the perturbation matrix $\bar A$ (from \eqref{eq:RR_svd_structure}) considered in our analysis. The operations (e.g. orthogonal multiplication) involved in both methods above are backward stable, so in finite precision arithmetic, we can consider that $\sigma_i$ is actually the exact singular value of $\bar A+E_u$ with $\|E_u\|_2=\mathcal{O}(u\|A\|_2)$. 
Therefore, Weyl's theorem 
implies that the contribution of round-off errors in $|\sigma_i-\theta_i|$ is fortunately bounded by $\mathcal{O}(u\|A\|_2)$. That is, even in finite precision arithmetic, the bounds can be trusted up to working precision.
    }
    \label{HMT_and_RR_power2}
\end{figure}

Looking at the results with both single and double power iteration allows us to compare the effect of a graded and ungraded residual structure on the bounds. With single power iteration we observe that the residuals are all around $10^{-4}$. Therefore the bound Li-Li (large gap) is almost as accurate as Li-Li (1-1 block), because $\|E_i\|_2 \approx \|E\|_2, \ \forall i$. In this situation the trend of the bounds is controlled by the spectral gap in the denominator.

With the double power iteration, the residuals have an exponential decay as we move towards the extreme eigenvalues, so naturally the bound Li-Li (1-1 block) and our bounds perform much better than the bounds based on $\|E\|_2$. Once again our bounds are sharper than Li-Li (1-1 block) by a factor of about $100$.

Hence both for single and double power iterations, our bound is the sharpest, to the point where it  
even has the same fluctuations as the actual error as the index $i$ is varied. Moreover the approximation $\mbox{Gap}_i= |\theta_i-\theta_{k}|$ leads to a bound that is very close to the theoretical result from \Cref{thm_svd} and uses only available information, so this bound is computable in practice. 

Also as mentioned for the asymptotic case at the end of \Cref{sec:svd_bd}, the bound \eqref{eq:liliSVDbd} (which is Li-Li (large gap)) is sharper than \eqref{eq:lorenzoSVDbd}. 
However, one must keep in mind that our bounds use stronger assumptions and information: 
while the bounds \eqref{eq:liliSVDbd} and \eqref{eq:lorenzoSVDbd} only use the spectral norm of the full error blocks $E$, $F$, we assume that all the norms of each column ($\|E_i\|_2, \|F_i\|_2$) are known. We repeat that these are indeed typically available or computable.

Finally, we briefly compare the results from HMT and PG. In theory, for the same left trial subspace $Q_1$, HMT gives more accurate singular values than PG, as HMT is based on an orthogonal projection. Experimentally we obtained nearly the same values of the residual norms $\|F_i\|_2$ with both methods (especially for small $i$), so our \Cref{thm_svd} suggests that the error bound for HMT is better than the bound for PG by about a factor 2. This is confirmed in \Cref{fig:pg_vs_hmt}. In all cases our bounds reflect the actual errors, and the PG estimates give accuracy comparable to that of HMT.

\begin{figure}
    \centering
    \includegraphics[width=0.7\linewidth]{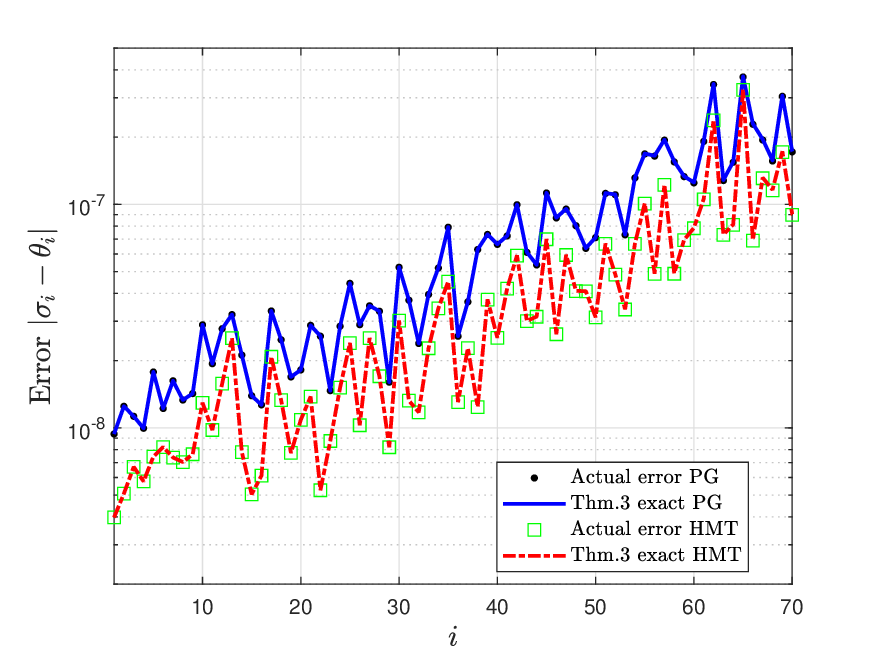}
    \caption{Error $|\theta_i-\sigma_i|$ and theoretical bound \Cref{thm_svd} for the Petrov-Galerkin and randomized SVD methods, computed based on the same left trial subspace $Q_1$ (with $Q_1$ obtained from a single power iteration).}
    \label{fig:pg_vs_hmt}
\end{figure}


\section{Discussion and future directions} 

One might wonder if the graded structure of the residuals could be exploited further in the derivation of our bounds. We discuss this below by looking at two cases.

To analyze the behavior of our bound when the residuals have a graded structure we look closer at the denominator of our bound, especially at the term $s_i:=\sum_{1\leq j\leq k \atop j\neq i} \|E_i\|_2^2/\delta_{i,j}$. We discuss the symmetric case for simplicity but the ideas also apply to the singular value case. The idea is to 
examine which conditions $s_i$ is small, as it corresponds to getting close to the asymptotic case. Take $i$ such that $1<i<k$ so that $\theta_i$ is somewhere in the middle of the approximated eigenvalues. Then in  terms of $s_i$ we consider a $j$ that is far from $i$ and distinguish roughly two cases :
\begin{itemize}
    \item[(i)] if $j> i$, then $|\theta_i-\theta_j|$ is large and $\|E_j\|_2\gg\|E_i\|_2$, therefore the fact that the denominator $|\theta_i-\theta_j|$ is large compensates for the large residual ;
    \item[(ii)] if $j<i$, then $|\theta_i-\theta_j|$ is large and $\|E_j\|_2\ll\|E_i\|_2$, which is the best case as it gives a small contribution to $s_i$. 
\end{itemize} 
Therefore for small $i$ most $j$ are larger so we are in the less favorable situation (i), but then it is compensated by the fact that the numerator $\|E_i\|_2^2$ of the bound is small. Conversely as $i\rightarrow k$ the numerator gets larger but the terms in $s_i$ become smaller so $d_i$ decreases. 

On the other hand, taking the $\theta_i$ to be uniformly distributed and the structure not graded we can assume $\|E_i\|_2\approx \epsilon= constant, \forall i$, which leads to $s_i \approx \epsilon^2\sum_{j\neq i} (|\theta_i-\theta_j|-\epsilon)^{-1}$ being almost independent of $i$. This implies that the bound is virtually only determined by $\mbox{Gap}_i=\min_j |\theta_i-\lambda_j(A_2)|$ and the residual $\|E_i\|_2^2$ (weak dependence on other terms $j\neq i$).

One way of looking at this is that our bounds have enough \textit{parameters} to account for a more structured error matrix than the previous bounds from the literature. Note that it also takes into account the other Ritz values $\theta_j$, although their influence is quite limited, so all the available information is used in the bound.

One might suspect that there is room for improvement in finding a more precise definition of $\delta_i, \delta_{ij}$ in order to make the bound sharper. 
However in practice the gap $\mbox{Gap}_i$ is the dominant term in the denominator of our bound: in our experiments from \Cref{sec:exp_sym} $\mbox{Gap}_i= O(1)$ while $\|E_i\|_2\approx 10^{-5}$ and $s_i\lesssim 10^{-6}$. 
Removing these terms and keeping only $\mbox{Gap}_i$ in the denominator kept our bound valid and did not change them much. 

Hence our theoretical work implies that, if the gap $\mbox{Gap}_i$ is known to be very large compared to the residual norms, one could even use the bound \eqref{eq:goal} with $c=1$ as an approximation to the error. Our result echoes the improved bounds for the accuracy of Ritz vectors that was derived in \cite{Yuji_sharpVect}, as it also improves a classical bound by showing that a bigger gap in the denominator can govern the perturbation of eigenvectors.

We believe that the most interesting extension of this work would be to derive a similar bound that is still applicable when the error matrix also has diagonal terms, or to find a trick that would put the error matrix in the right form in order to apply our bound. This could notably allow us to derive sharp error bounds for the approximate singular values computed from the (Generalized) Nyström method, which is often seen to be more accurate than Petrov-Galerkin and the randomized SVD methods. 
A somewhat related problem is to derive error bounds for eigenvalues computed by the \emph{sketched} Rayleigh-Ritz method~\cite{nakatsukasa_tropp_2024fast}, which is also not based on orthogonal projection. 
Another idea could be finding block-wise bounds inspired by our Thm.\ref{thm:main} that would be adapted to specific eigensolvers. One can notably think of Krylov methods with restarting (see for example the recent bounds from \cite{restart}), which could lead to a more complex structured perturbation matrix, but this idea is out of the scope of this paper. 

\subsection*{Acknowledgment}
We would like to thank the reviewers for their perceptive comments and suggestions, Andrew Knyazev for bringing several papers to our attention, and Ren-Cang Li for a discussion about how the results compare with~\cite{LiLi}.
This project was supported by the EPSRC grants EP/Y010086/1 and EP/Y030990/1. 
For the purpose of open access, the author has applied a CC BY public copyright licence to any author accepted manuscript arising from this submission.

\bibliographystyle{siamplain}
\bibliography{main}
\end{document}